\documentclass[12pt]{article}
\usepackage{amssymb}
\usepackage{amsthm}
\usepackage{amsmath}
\usepackage{amscd}
\usepackage{array}
\usepackage[curve]{xypic}
\usepackage{hyperref}
\usepackage[latin2]{inputenc}
\usepackage{t1enc}
\usepackage{enumerate}
\usepackage[mathscr]{eucal}
\usepackage[british]{babel}
\usepackage[dvips]{graphicx}
\usepackage[dvips]{epsfig}
\usepackage{epsf}
\usepackage{indentfirst}
\usepackage{url}
\title{\textbf{$K_0$-invariance of the completely faithful property of Iwasawa modules}}
\author{Tamas Csige \thanks{The author is a PhD student at E\"otv\"os Lor\'and University, Budapest}}

\newcommand{\defeq}{\mathrel{\mathop:}=}
\newtheorem{thm}{Theorem}[subsection]
\newtheorem{pro}[thm]{Proposition}
\newtheorem{lem}[thm]{Lemma}

\newtheorem{rem}[thm]{Remark}

\theoremstyle{definition}
\begin{document}
\vspace{2cm}
\maketitle
\vspace{1cm}
\begin{abstract} \noindent Let $H$ be a compact $p$-adic analytic group without torsion element, whose Lie algebra is split semisimple and $\mathfrak{N}_H(G)$ be the full subcategory of the category of finitely generated modules over the Iwasawa algebra $\Lambda_G$ that are also finitely generated as $\Lambda_H$-modules, where $G = \mathbb{Z}_{p} \times H$. We show that if the class of a module $N$ in the Grothendieck group of $\mathfrak{N}_H(G)$ equals to the class of a completely faithful module, then $q(N)$ is also completely faithful, where $q(N)$ denotes the image of $N$ via the quotient functor modulo the full subcategory of pseudonull modules. We also generalize a Theorem of Konstantin Ardakov characterizing the completely faithful property to the case of more general $p$-adic Lie groups.
\vspace{0.75cm}


\end{abstract}
\vspace{0.75cm}
\section{Introduction} In \cite{C} the authors state the $GL_2$ main conjectures for elliptic curves over $\mathbb{Q}$. For such an elliptic curve set $F_{\infty} = \mathbb{Q}(E[p^{\infty}])$, i.e. the extension of $\mathbb{Q}$ with the coordinates of the $p$-division points od $E$. By Weil pairing, $\mathbb{Q}(\mu_{p^{\infty}}) \subset F_{\infty}$, hence $F_{\infty}$ contains $\mathbb{Q}^{\textnormal{cyc}}$. Now set $G = \textnormal{Gal}(F_{\infty}/\mathbb{Q})$ and $H = \textnormal{Gal}(F_{\infty}/\mathbb{Q}^{\textnormal{cyc}})$, we have that $Z = G/H \simeq \mathbb{Z}_p$. One is particulary interested in the Pontrjagin-dual of the Selmer group of $E$, i.e. $X(E/F_{\infty}) = \textnormal{Hom}(\textnormal{Sel}(E/F_{\infty}),\mathbb{Q}_p/\mathbb{Z}_p)$. There is an important left and right Ore set in $\Lambda_G$, $S^{*}$ and one can define the category of finitely generated $S^*$-torsion $\Lambda_G$-modules $\mathfrak{M}_H(G)$ (for details see \cite{C} Chapter $2$).The first conjecture in \cite{C} states that under suitable assumtions, $X(E/F_{\infty})$ is an object of  $\mathfrak{M}_H(G)$. It also can be shown that a module $M \in \mathfrak{M}_H(G)$ if and only if $M/M(p)$ is finitely generated over $\Lambda_H$, where $M(p)$ denotes the $p$-primary submodule of $M$. Now we see that the category $\mathfrak{N}_G(H)$ contains all the quotient modules of the form $M/M(p)$, so it gives us a functor $F$, from the category $\mathfrak{M}_H(G)$ to $\mathfrak{N}_H(G)$ and it is also true that $\mathfrak{N}_H(G)$  is a full subcategory of $\mathfrak{M}_H(G)$ and so we have a natural functor $\mathfrak{N}_H(G) \rightarrow \mathfrak{M}_H(G)$ and it is exact. Hence it induces a map $\varphi:K_0(\mathfrak{N}_H(G)) \rightarrow K_0(\mathfrak{M}_H(G))$. If $G$ has no element of order $p$,  theorem $2.1$ in \cite{BV} shows that $K_0(\mathfrak{N}_H(G))$ is a direct summand of $K_0(\mathfrak{M}_H(G))$. 
So we have a map in the reverse direction and it is actualy induced by the functor $F$ defined above.
\bigskip

\noindent If we assume again that $G$ has no element of order $p$, then mainly as a consequence of Quillen's theorem on localization sequence, we have a  map \[ \partial_G : K_1((\Lambda_{G})_{S^*}) \rightarrow K_0(\mathfrak{M}_H(G)).\] It was shown in \cite{C} that this map is surjective.
One can define the characteristic element of a module $M$ as an inverse image of $[M]$ with respect to this surjective map, so if the conjecture stated above is true, we can define the characteristic element of the dual of the Selmer group of $E$. The main conjecture states that if other conjectures hold (one is what we already mentioned and the other is that the $p$-adic $L$-function, $\mathcal{L}_E$ exists in $K_1((\Lambda_{G})_{S^*})$), then $\mathcal{L}_E$ is a characteristic element of the dual, i.e. of $X(E/F_{\infty})$.
\newline Knowing that the competely faithful property is $K_0$-invariant in $K_0(\mathfrak{N}_H(G))$ in the sense of above can bring us closer to determine the completely faithful property via the characteristic element  in $K_1(\Lambda(G)_{S*})$. In \cite{Z} the authors find examples where the dual of the Selmer group is completely faithful. The question whether or not this is true for any ellipitic curve was raised in \cite{CS}.
\newline As a side effect of our investigation we give a slight generalization of a theorem of Konstantin Ardakov (see \cite{Ar1} Theorem $1.3$), namely instead of $G=Z \times H$ where $H$ is an open subgroup without torsion element and its Lie-algebra is split, semisimple we let its Lie-algebra be the direct product  of a split, semisimple and an abelian Lie-algebra.  

\section{Preliminaries}
\subsection{Iwasawa algebras.} \noindent Let $p$ be a prime number. We will work with modules over Iwasawa algebras
\begin{center} $\Lambda_G \defeq \varprojlim_{N \lhd_o G} \mathbb{Z}_p[G/N]$
\end{center} 
where $G$ is a compact $p$-adic analytic group.
\vspace{0.5cm}

 \noindent Iwasawa theory for elliptic curves in arithmetic geometry provides the main motivation for the study of Iwasawa algebras, for example when G is a certain subgroup of the $p$-adic analytic group $\textnormal{Gl}_2(\mathbb{Z}_p)$. Homological and ring-theoretic properties of these Iwasawa algebras are useful for understanding the structure of the Pontryagin dual of Selmer groups and other modules over the Iwasawa algebras. For more information, see for example \cite{AB} or \cite{CS}.
\vspace{0.5cm}
\subsection{\textnormal{\textbf{Fractional ideals and prime c-ideals}}}
\subsubsection{\textnormal{\textbf{Fraction Ideals}}}\noindent Let $R$ be a Noetherian domain, then it is well-known that $R$ has a skewfield of fractions, $Q(R)$. Recall that a right $R$-submodule $I$ of $Q(R)$ is called fractional right $R$-ideal if it is non-zero and there is a  $q \in Q(R)$, such that $q\neq 0$ and $I \subseteq qR$.  One can define fractional left $R$-ideals similarly. When it is obvious what ring we mean, we just call it a fractional right or left ideal. If we have a fractional right ideal $I$, one can define its inverse \[ I^{-1} := \{ q \in Q(R) \ | \ qI \subseteq R\}\]
\vspace{0.3cm}
Of course we can define the inverse for fractional left ideals. Let us consider the dual of $I$, i.e. $I^* = \textnormal{Hom}_{R}(I,R)$. This is a left $R$-module and there is a natural isomorphism $u:I^{-1} \rightarrow I^*$ that sends an element $i \in I^{-1}$ to the homomorphism induced by left multiplication by $i$. The following lemma is useful to compute $I^{-1}$.
\begin{lem}\textnormal{ Let $R$ be a Noetherian domain and $I$ be a non-zero right ideal of $R$. Then $I^{-1}/R \simeq \textnormal{Ext}^1(R/I,R)$.}
\end{lem}
\subsubsection{\textnormal{\textbf{Prime c-ideals}}}
\noindent Let $I$ be a fractional right ideal. The reflexive closure of $I$ is $\overline{I} :=(I^{-1})^{-1}$. This is also a fractional right ideal and it contains $I$. $I$ is called reflexive if it is the same as its reflexive closure, i.e. $I = \overline{I}$. One can say equivalently that $I \rightarrow (I^*)^*$ is an isomorphism. The next proposition will be quite useful.
\begin{pro} \textnormal{Let $R \hookrightarrow S$ be a ring extension such that $R$ is noetherian and $S$ is flat as a left and right $R$-module. Then there is a natural isomorphism \[ \psi_M^i : S \otimes_R \textnormal{Ext}_R^i(M,R) \rightarrow \textnormal{Ext}_S^i(M \otimes_RS,S)\]
for all finitely generated right $R$-modules and all $i \geq 0$. A similar statement holds for left $R$-modules. If moreover $S$ is a noetherian domain, then 
\newline \noindent (i) $\overline{ I \cdot S} = \overline{I} \cdot S$ for all right ideals $I$ of $R$,
 \noindent (ii) if $J$ is a reflexive right $S$-ideal, then $I \cap R$ is a reflexive right $R$-ideal.} 
\end{pro}
\begin{proof} We omit the proof of this proposition. It can be found in \cite{Ar2} Proposition $1.2$.
\end{proof}
\noindent Another important result is the following
\label{P}\begin{pro}\textnormal{Let $R$ be a noetherian domain and $I$ be a proper $c$-ideal of $R$. If there is an element $x \in R$ such that $x$ is non-zero, central in $R$, $R/x \cdot R$ is a domain and $x \in I$ then $I = x \cdot R$.}
\end{pro}
\begin{proof} For proof see for example \cite{Ar1} Lemma $2.2$.
\end{proof}
\vspace{0.3cm}
 \subsection{\textnormal{\textbf{Pseudo-null modules}}}
\noindent Let $R$ be an arbitrary ring and $M$ be an $R$-module. Recall from \cite{Ar2} (section $1.3$) that $M$ is called  pseudo-null if $\textnormal{Ext}_R^0(N,R) = \textnormal{Ext}_R^1(N,R) = 0$ for any submodule $N \subseteq M$. The category of pseudo-null modules $\mathcal{C}$ is a full subcategory of \textnormal{mod}(R), i.e. the category of finitely generated right $R$-modules and is "localizing" in the sense that it is a Serre subcategory, i.e. if  $0 \rightarrow A_1 \rightarrow A \rightarrow A_2 \rightarrow 0$ is a short exact sequence of right $R$-modules, then $A$ lies in $\mathcal{C}$ if and only if $A_1$ and $A_2$ lie in $\mathcal{C}$ and any $R$-module has a largest uniquesubmodule contained in $\mathcal{C}$. Let $R$ be a noetherian domain. Since $\mathcal{C}$ is a Serre subcategory we can consider the quotient category and the quotient functor \[q: \textnormal{mod}(R) \rightarrow \textnormal{mod}(R)/\mathcal{C}. \] 
\label{C}\subsection{Completely faithful modules}\noindent We recall from \cite{Ar2} what a completely faithful module is. One can consider the annihilator ideal of an object $M$ in the quotient category \[ \textnormal{Ann}(M) :=\sum \{ \textnormal{Ann}_R(N) \ | \ q(N) \simeq M\}\] and $M$ is said to be completely faithful if Ann($L$) = $0$ for any non-zero subquotient $L$ of $M$. Now let $G = H \times Z$ where $Z = \mathbb{Z}_p$ and $H$ a compact $p$-adic analytic group without torsion element, whose Lie algebra is split semisimple. Let mod($\Lambda_G$) be  the category of finitely generated left $\Lambda_G$-modules and consider $\mathfrak{N}_H(G)$, the full subcategory of all left $\Lambda_G$-modules that are finitely generated as $\Lambda_H$-modules. \newline \noindent Note that if $M \in \mathfrak{N}_H(G)$ then it has no pseudo-null submodules if and only if it has no $\Lambda_H$ torsion element.  Recall  \cite{Ar1} Theorem 1.3 for it is of particular importance:
\label{A}\begin{thm}\textnormal{Let $p \geq 5$ and $G$ as above. If $M$ has no non-zero pseudo-null submodules, then $q(M)$ is completely faithful if and only if $M$ is torsion-free over $\Lambda_Z$}.
\end{thm}
\vspace{0.5cm}
\subsection{The Grothendieck group of $\mathfrak{N}_H(G)$}
\noindent We will need  the definition of the Grothendieck group $K_0(\mathcal{A})$ of a skeletally small abelian category from \cite{B} (Definition $6.1.1$) and some results in connection with it. However we will have a module category $\mathfrak{N}_H(G)$ in which case $K_0(\mathfrak{N}_H(G))$ is the abelian group presented as having one generator $[A]$ for each isomorphism class of modules and one relation for every short exact sequence, i. e. $[A_2] = [A_1] + [A_3]$ in $K_0(\mathcal{A})$ whenever we have a short exact sequence of the form
\begin{center}
$\begin{CD}
0 @>>> A_1 @>>> A_2 @>>> A_3 @>>> 0
\end{CD}$
\end{center}
where $A_1, A_2, A_3 \in \mathfrak{N}_H(G)$. We will make use of a fact in K-theory that basically tells us how many relations we have whenever two modules have the same class in $K_0(\mathfrak{N}_H(G))$. \noindent  However we will state a more general version. For details see \cite{B} (Ex. $6.4.$)
\label{W}\begin{lem} \textnormal{ Let $\mathcal{A}$ be a small abelian category. If $[A_1] = [A_2]$ in $K_0(\mathcal{A})$ then there are short exact sequences in $\mathcal{A}$
\vspace{0,5cm}
\begin{center}
\label{1}$\begin{CD} 0 @>>> C @>>> K @>>> D @>>> 0 \\
0 @>>> C @>>> L @>>> D @>>> 0
\end{CD}$
\end{center}
\vspace{0,3cm}
such that $A_1 \oplus K = A_2 \oplus L$.}
\end{lem} 
\vspace{0.5cm}
\subsection{The completely faithful property and $K_0$ invariance}
\label{CS}\begin{thm}\textnormal{ Let $p\geq 5$. Let $H$ be a torsion-free compact $p$-adic analytic group whose Lie algebra is split semisimple over $\mathbb{Q}_p$ and let $G=\mathbb{Z}_{p} \times H$. Let $M$, $N \in \mathfrak{N}_H(G)$, both $\Lambda_H$-torsion free and let $q(M)$ be completely faithful in the sense of \ref{C}. If  $[M] = [N]$ in $K_0(\mathfrak{N}_H(G))$ then $q(N)$ is also completely faithful.}
\end{thm}
\begin{rem}\textnormal{ One can think of $G$ as $\Gamma_1$ which is the first inertia subgroup of $\textnormal{GL}_n(\mathbb{Z}_p)$ i.e.
\vspace{0.5cm}
\begin{center} $\Gamma_1 = \{ \gamma \in \textnormal{GL}_n(\mathbb{Z}_p) | \gamma \equiv 1 \mod (p)\}$
\end{center}
\vspace{0,5cm}
 \noindent In this case $G = Z \times H$ where $Z \simeq \mathbb{Z}_p$ is the centre of $G$ and $H$ is an open subgroup of $\textnormal{Sl}_n(\mathbb{Z}_p)$ that is normal in $G$.}
\end{rem}
 \subsection{Some well known facts in commutative algebra}
\label{comm}\noindent We will briefly mention some additional tools we use in order to prove \ref{CS}. First suppose $R$ is a commutative ring, the support of an $R$-module $M$ denoted by $\textnormal{Supp}_R(M)$ is the set of prime ideals, $P$ of $R$ such that the localized module $M_P \neq 0$. If $M$ is finitely generated then $\textnormal{Supp}_R(M)$ is exactly the set of prime ideals containing $\textnormal{Ann}_R(M)$. 
\newline \noindent  Note that $M$ is torsion-free if and only if $M$ has no $N \leq M$ $R$-submodule such that $\textnormal{Ann}_R(N) \neq 0$.
\newline \noindent We use the usual notation for the set of all prime ideals of a ring $R$ by Spec($R$).
\newline \noindent It is also well known that the nilradical is the set of nilpotent elements and also the intersection of all prime ideals of $R$.
\section{Theorem \ref{CS}}
\noindent Now we are ready to prove what we stated in section \ref{CS}.
\subsection{The proof Theorem \ref{CS}} 
\begin{proof} By Theorem \ref{A} it is enough to prove that $N$ is $\Lambda_Z$ torsion-free. In order to do that recall from \ref{comm} that it suffices to show that $\textnormal{Ann}_{\Lambda_Z}(N') = 0$ for all $N'$ $\Lambda_G$-submodule of $N$. It is because if there is a $\Lambda_Z$-submodule $N'$ of $N$ (which is naturaly a $\Lambda_Z$-module) with $\{n_1,\dots \}$ a set of generator system, then when we take the module $\overline{N'}$ generated by the same set of elements $\{n_1,\dots\}$ as a $\Lambda_G$-module, it is still in $N$, so it will be a $\Lambda_G$-submodule. If there is a non trivial element in $\textnormal{Ann}_{\Lambda_Z}(N')$, then it still annihilates all the elements of $\overline{N'}$, since $\Lambda_Z$ is central in $\Lambda_G$.
\newline \noindent First we show  that $\textnormal{Ann}_{\Lambda_Z}(N) = 0$. The $\Lambda_G$-module $N$ might not be finitely generated as a $\Lambda_Z$-module but it is still true that all $P \in \textnormal{Supp}_{\Lambda_Z}(N)$ contains $\textnormal{Ann}_{\Lambda_Z}(N)$, since if $N_P \not =0$ then there is an element of the generator system (we choose one, for example $\{n_1,n_2, \dots\}$), e.g. $n_1$ which is not zero after localizing, so $\bigcap \textnormal{Ann}(n_i) = \textnormal{Ann}(N) \subseteq \textnormal{Ann}(n_1) \subseteq P$. \newline \noindent It is known that the nilradical of $\Lambda_Z$ is zero. Moreover since $\Lambda_Z$ is a commutative noetherain maximal order, it is just a noetherian integrally closed domain, see \cite{AB} $3.6$.  The Nilradical is also the intersection of all prime ideals of $\Lambda_Z$, so it is enough to prove that $\textnormal{Supp}_{\Lambda_Z}(N) = \textnormal{Spec}(\Lambda_Z)$. 
\newline \noindent For that we localize in $\Lambda_H$ at the $(0)$ ideal. By Goldie's theorem we get a skewfield what we denote by $Q(H)$. If $S$ is a $\Lambda_G$-module from the category $\mathfrak{N}_H(G)$ then after localization we get a finite dimensional vector space, $Q(S)$ over $Q(H)$. Now recall that by Lemma \ref{W} we have short exact sequences \ref{1} such that $M \oplus K = N \oplus L$. Since localization is exact and $\Lambda_Z$ is central in $\Lambda_G$ the statement of Lemma \ref{W} is still true, so we still have the exact sequences and the equation of direct sums of finite dimensional $Q(H)$ vector spaces that are also $\Lambda_Z$-modules. Lets suppose indirectly that there is a $P \in \textnormal{Spec}(\Lambda_Z)$ such that $N_P = 0$. Localize with $P$ in $\Lambda_Z$. Since localization is exact and $\Lambda_Z$ is central in $\Lambda_G$ we get short exact sequences
\vspace{0.5cm}
\begin{center} 
$\begin{CD} 0 @>>> Q(C)_P @>>> Q(K)_P @>>> Q(D)_P @>>> 0 \\
0 @>>> Q(C)_P @>>> Q(L)_P @>>> Q(D)_P @>>> 0
\end{CD}$
\end{center}
\vspace{0,5cm}
such that $Q(M)_P \oplus Q(K)_P = Q(N)_P \oplus Q(L)_P$.
\label{CS2}\begin{lem}\textnormal{Let $Q(M)$ be a finite dimensional vectorspace over $Q(H)$ with a $\Lambda_Z$ action on it and let $P$ be an arbitrary prime ideal of $\Lambda_Z$. Then $Q(M)_P$ is also finite dimensional over $Q(H)$ where $Q(M)_P$ denotes the localized module of $Q(M)$ with $P$.} 
\end{lem}
\begin{proof} Let $Q(M)_S$ be the $S$-torsion submodule of $Q(M)$ where $S = \Lambda_Z \setminus P$. This is a $Q(H)$ subspace of $Q(M)$ since $\Lambda_Z$ is central and the set $S$ is multiplicately closed. Also $Q(M)_S \cdot S^{-1} = 0$. One can see that the quotient submodule $Q(M)/Q(M)_S$ is $S$ torsion-free since $s \cdot m \in Q(M)_S \Leftrightarrow \exists s_1 \in S$ such that $s_1\cdot s\cdot m = 0$ but that implies that $m \in Q(M)_S$. 
\newline \noindent We have a short exact sequence of $\Lambda_G$-modules
\begin{center}$\begin{CD} 0 @>>> Q(M)_S @>>> Q(M) @>>> Q(M)/Q(M)_S @>>> 0
\end{CD}$
\end{center}
\vspace{0,5cm} After localization with $P$ one can easily see that $Q(M) \cdot S^{-1} \simeq (Q(M)/Q(M)_S) \cdot S^{-1}$.
We will prove that $(Q(M)/Q(M)_S) \cdot S^{-1}  \simeq Q(M)/Q(M)_S$ as $Q(H)$ vector spaces. We can consider localization of a module as tensoring it by the localized ring so one can deduce that the elements of $(Q(M)/Q(M)_S) \cdot S^{-1}$ are of the form $(x \otimes \frac{1}{s})$ where $x \in Q(M)/Q(M)_S$ and $s \in S$.  Let us observe first that multiplication with and arbitrary element $s \in S$ is an injective linear transormation on the vector space $Q(M)/Q(M)_S$ hence it is an isomorphism so one can see that using the surjectivity property that every $x \in Q(M)/Q(M)_S$ can be written as $s \cdot y$ for some $y \in Q(M)/Q(M)_S$ so every element is of the form $(y \otimes 1)$ in $Q(M)/Q(M)_S \cdot S^{-1}$.  Since both $Q(M)$ and $Q(M)_S$ are finite dimensional, it follows that $Q(M)/Q(M)_S$ is also finite dimensional. Let $e_1, \dots, e_n$ be a basis in $Q(M)/Q(M)_S$. We have just seen that $(e_i \otimes 1)$ is a generating system for $Q(M)/Q(M)_S \cdot S^{-1}$. In order to see that they are linearly independent one easily see that if there is a linear combination of these is also of the form $(y \otimes 1)$, it is zero if and only if $y$ is $S$-torsion, but $Q(M)/Q(M)_S$ is $S$-torsion free. 
\end{proof}
\vspace{0.3cm} \noindent So all the the modules in the short exact sequences remain finite dimensional over $Q(H)$ after localization with $P$ by Lemma \ref{CS2}. Then by the indirect statement above, we see that $Q(L)_P = Q(M)_P \oplus Q(K)_P$ so we have 
\vspace{0.5cm}
\begin{center}$\begin{CD} 0 @>>> Q(C)_P @>>> Q(K)_P @>>> Q(D)_P @>>> 0 \\
0 @>>> Q(C)_P @>>>Q(M)_P \oplus Q(K)_P @>>> Q(D)_P @>>> 0
\end{CD}$
\end{center}
\vspace{0,5cm}
but that since these are finite dimensional vector spaces over $Q(H)$ we have that $Q(M)_P = 0$ which cannot be since $M$ is completely faithful. 
\newline \noindent So now we see that $\textnormal{Ann}_{\Lambda_Z}(N) = 0$. We are left to prove that $N$ has no $N'$ $\Lambda_G$-submodule such that $\textnormal{Ann}_{\Lambda_Z}(N') \neq 0$. Let us suppose that there is one. It means that there is a $P \in \textnormal{Spec}(\Lambda_Z)$ prime ideal such that $Q(N')_P =0$. By using Lemma \ref{W} again we see that $\textnormal{dim}_{Q(H)}Q(M)_P< \textnormal{dim}_{Q(H)}Q(M)$. It is because there is vector subspace that vanishes by the indirect statement. We need one more lemma.
\begin{lem}
\textnormal{  Let us suppose that a $\Lambda_G$-module $M$ is $\Lambda_H$-torsion free. $M$ is $\Lambda_Z$-torsion free if and only if $Q(M)$ is $\Lambda_Z$-torsion free.
}
\end{lem}
\begin{proof} Let us suppose first that $Q(M)$ is $\Lambda_Z$-torsion free. We deduce that $M$ cannot have a $\Lambda_Z$-torsion part. It is because $\Lambda_Z$ is central in $\Lambda_G$, so if  $m$ is a $\Lambda_Z$-torsion element, all the elements $m/s$ are also torsion elements in $Q(M)$, so we get a $\Lambda_Z$-torsion part in $Q(M)$ (which is the localized submodule of the $\Lambda_Z$-torsion part of $M$).   
\newline \noindent The other direction can be proved by the following: Let us suppose that $M$ is $\Lambda_Z$-torsion free and indirectly that there is a $\Lambda_Z$-torsion part of $Q(M)$. That means that there exist elements in $Q(M)$ such that for each of them there is a $z \in \Lambda_Z$ such that $m/sz = 0$ in $Q(M)$. By the definition of localization we see that there are elements $s_1,s_2 \in \Lambda_H - \{ 0 \}$ such that $(mzs_1 - 0s)s_2 = 0$ in $M$. Hence $ms_1s_2z = 0$ ($\Lambda_Z$ is central), but $M$ is $\Lambda_H$-torsion free hence $z$ annihilates $ms_1s_2$, but that cannot be since $M$ is $\Lambda_Z$-torsion free.
\end{proof}
 \noindent Now we are ready to finish the proof. By Theorem $2.4.1$ (\ref{A}) we see that $M$ is $\Lambda_Z$-torsion free. Hence $Q(M)$ has the same property by the lemma above. The natural map $Q(M) \rightarrow Q(M)_P$ is therefore injective since the kernel of this map consists of $\Lambda_Z$-torsion elements in $Q(M)$. From the indirect statement we deduced that $\textnormal{dim}_{Q(H)}Q(M)_P< \textnormal{dim}_{Q(H)}Q(M)$, but that cannot happen by the injectivity we showed just now.
\end{proof}
\section{Theorem \ref{A}}
\subsection{Generalization of Theorem \ref{A}}
\noindent In \cite{Ar1} the author proved that whenever $G = Z \times H$ where $Z = \mathbb{Z}_p$ and $M$ is finitely generated $\Lambda_G$-module which has no non-zero pseudo-null submodule then $q(M)$ is completely faithful if and only if $M$ is $\Lambda_Z$ torsion-free. We will generalize it to $G =  Z_1 \times H_1$ where $Z_1 \simeq \mathbb{Z}_p$ and $H_1$ is such that its Lie-algebra is the direct product of a split semisimple and an abelian Lie-algebra. In this case $H_1 = \mathbb{Z}_p^n \times H$ where $H$ is such that its Lie-algebra is split semisimple. Let us denote $Z_2 = \mathbb{Z}_p^n$. We can choose topological generator for $\mathbb{Z}_p \simeq Z_1$ i.e. $g_1$ such that $<\overline{g_1}>\simeq \mathbb{Z}_p$. Now let $z = g_1 -1$. Let $\Lambda_Z$ be the Iwasawa algebra over the group $Z = Z_1 \times Z_2$.
\label{CS1}{\begin{thm}\textnormal{ Let $p \geq 5$. Let  $G = Z_1 \times H_1$ where $Z_1$ and $H_1$ as above. Let $M$ be a finitely generated torsion $\Lambda_G$-module such that it has no non-zero pseudo-null submodules. Then $q(M)$ is completely faithful if and only if $M$ is $\Lambda_Z$ torsion-free.}
\end{thm}}
\label{P1}\begin{pro}\textnormal{Let $G$ be a torsionfree compact $p$-adic analytic group. Then $\Lambda_G$ is a maximal order.}
\end{pro}
\proof{see \cite{Ar1} Theorem $4.1.$}
\vspace{0,3cm}
 \noindent  We will use two localization of $\Lambda_G$ in order to prove the next proposition. Let $S$ be the Ore set
\begin{center} $S = \{ s \in \Lambda_G \ | \  \textnormal{s is regular mod} \ P_H\}$
\end{center}
and let $T$ be another Ore set such that
\begin{center} $T =\{ t \in \Lambda_G \ | \ \Lambda_G/\Lambda_G \cdot t \ \textnormal{finitely generated over} \  \Lambda_{H_1}\}$
\end{center}
\vspace{0.3cm}
The definition of these two Ore sets can be found in \cite{Ar1} $3.1$ and in \cite{C} Definition $2$ respectively, altough both are the work of Venjakob. By Propositsion $2.6$ in \cite{C} one can identify $T$ to be those elements that are regular modulo $\mathcal{N}$  where \[ \mathcal{N} :=\textnormal{preimage of} \ \mathcal{N}(\Omega(G/J)) \ \textnormal{in} \ \Lambda_G\] and $J$ denotes an arbitrary pro-p subgroup of $H$ which is normal in $G$. For details see \cite{C} section $2$. \newline \noindent The special case of the next proposition can be found in \cite{Ar1}Theorem  $4.3$
\label{Q}\begin{pro}\textnormal{If $I$ is a   prime $c$-ideal of $\Lambda_G$ then it is $p \cdot \Lambda_G$ or $\Lambda_G/I$ is finitely genrated over $\Lambda_{H_1}$ }
\end{pro}
\begin{proof} First case: If $I \cap T = \emptyset$. It can be seen that $S \subseteq T$ so that means $I \cap S = \emptyset$ also. So by Propositsion $3.4$ in \cite{Ar1} the localized ideal $I_{S}$ in $\Lambda_{G, H}$ equals to $p \cdot \Lambda_{G, H}$, but it means that $p \in I = I_S \cap \Lambda_G$ so by Propositsion \ref{P} $I = p \cdot \Lambda_G$.
\newline \noindent Second case: If $I \cap T \neq \emptyset$. It means that $\Lambda_G / I$ is $T$-torsion in the sense of \cite{C} (section $2$). So by Propositsion $2.3$ in \cite{C} $\Lambda_G / I$ if finitely generated over $\Lambda_{H_1}$. 
\end{proof}
\noindent The next proposition is of crucial importance.
\label{R}\begin{pro}
\label{S}\textnormal{If $I$ is a prime $c$-ideal in $\Lambda_G$ such that $\Lambda_G/I$ is finitely generated over $\Lambda_{H_1}$. Then $\Lambda_Z \cap I \neq \emptyset$ or $I \cap \Lambda_{H_1} \neq \emptyset$.}
\end{pro}
\begin{proof} The proof of this proposition is quite similar to that of in \cite{Ar1} Prop. $4.5$ with a little additional argument. We again have an increasing chain of finitely generated $\Lambda_{H_1}$-modules \[\Lambda_{H_1} = A_0 \subset A_1 \subset A_2 \dots \]
where $A_i = \bigoplus^{i}_{k = 0} \Lambda_{H_1}z^k$. The image of this chain in $\Lambda_G/I$ must stabilize by the Noetherian property of $\Lambda_G/I$ that we assumed. So again $I \cap A_n \neq 0$ for some $n$. Let us consider the minimal such $n$, if it is zero, then we are done, if it is not, then we have a non-zero polynomial \[ a = a_nz^n + \dots + a_0 \in I. \] By minimality of $n$ $a_n \neq 0$. Writing again $Q(H)$ for the skewfield of fractions of $\Lambda_{H_1}$ we consider the polynomial ring $Q(H_1)[z]$. Note that $Q(I \cap \Lambda_{H_1})$ is a two-sided ideal in $Q(H_1)[z]$ and $a_n^{-1}a \in Q(I  \cap \Lambda_{H_1}[z])$. Consider an element $u \in Q(H_1)$ and look at the commutator $[u,a_n^{-1}a]$. It has strictly smaller degree than $n$ and it is still in the ideal $Q(I \cap \Lambda_{H_1})$. So with clearing the common denominator we get an element which is in $I \cap A_{n-1}$, so it must be zero by minimality of $n$. It means that \[a_n^{-1}a_i \ Z(Q(H_1)) \]  for all $i < n$. Since $\Lambda_{Z_2}$ is central in $\Lambda_{H_1}$ and the center of $\Lambda_H$ is $\mathbb{Q}_p$ by Theorem $4.4$ in \cite{Ar1}, it follows that $a_n^{-1}a \in Q(Z_2)$ i.e. in the field of fractions of $\Lambda_{Z_2}$ therefore we can find an $m \geq 0$ such that $p^m a_n^{-1}a \in \Lambda_Z$. Now $I$ is a prime ideal, $a_n$ is not in $I$ and $p^m a = a_n (p^m a_n^{-1}) a Q \in I$, we see that $I \cap \Lambda_Z \neq \emptyset$. 
\end{proof}
\noindent Now we are ready to prove Theorem $4.1$. It is essentialy the same as the proof of the special version in \cite{Ar1} with a little modification.
\vspace{0.5cm}
\subsection{Proof of Theorem \ref{CS1}}
\begin{proof}First by Propositsion \ref{P1} $\Lambda_G$ is a maximal order. By Proposition $4.1.1.$ in \cite{CS} and the fact that $M$ is $\Lambda_G$ torsion, $q(M) = M_0 \oplus M_1$ where $M_0$ is completely faithful and $M_1$ is locally bounded. Let us consider $X \leq M$ the $\Lambda_Z$-torsion submodule of $M$. Like in \cite{Ar1} we show that $q(X) = M_1$. Since $\Lambda_G$ is noetherian we can find a finitely generated $\Lambda_G$-submodule $\mathcal{N}$ of $M$ such that $q(\mathcal{N}) = M_1$. Let $M_o$, $\mathcal{N}_o$ be the maximal pseudo-null submodule of the modules $M$ and $\mathcal{N}$. Since $\mathcal{N}_o \leq M_o = 0$, the annihilator of $M_1$ equals to the annihilator of $\mathcal{N}$. Since $M_1$ is locally bounded, $\mathcal{N}$ is a $\Lambda_G$-torsion, bounded object in $\textnormal{mod}(\Lambda_G)$. So by Lemma $4.3$ (i) in \cite{CS} $\textnormal{Ann}(q(\mathcal{N}))$ is a prime-c ideal. Therefore by Propositsion \ref{S} there is a central element $x \in \Lambda_{Z}$ such that $x \in \textnormal{Ann}(q(\mathcal{N}))$, so $x \Lambda_G \subseteq \textnormal{Ann}(q(\mathcal{N}))$ but it means that $\mathcal{N} \subseteq X$. Hence $M_1 = q(\mathcal{N}) \subseteq q(X)$ but $q(X) \subseteq M_1$ is true also since $\Lambda_Z$ is central. Of course $q(M)$ is completely faithful if and only if $M_1 = 0$ but now we see that it happens if and only if $\mathcal{N} = 0$ since $\mathcal{N}$ has no non-zero pseudo-null submodule.
\end{proof}


\vspace{1.5cm}
\begin{thebibliography}{99}
\bibitem{Ar1} K. \ Ardakov, Centres of Skewfields and completely faithful
Iwasawa modules. \emph{ J. Inst. Math. Jussieu 7} (2008).
\bibitem{Ar2}  K. \ Ardakov, F. \ Wei, J. \ J. \ Whang,  Reflexive ideals in Iwasawa algebras. \emph{ Adv. Math.} \textbf{218} (2008), 865-901.
\bibitem{AB} K. \ Ardakov, K. A. Brown, Ring-theoretic properties of Iwasawa algebras: a survey, \emph{Documenta Math., Extra volume Coates}, (2006), 7-33.
\bibitem{Z} T. \ Backhausz, G. \ Z\'abr\'adi, Algebraic functional equations and completely faithful Selmer groups. \url{ http://arxiv.org/pdf/1405.6180v1.pdf}
\bibitem{BV} D. \ Burns, O. Venjakob, On descent theory and main conjectures of non-commutative Isawawa theory. 
\url{https://www.mathi.uni-heidelberg.de/~venjakob/papervenjakob/finaldescent.pdf }
\bibitem{CS} J. \ Coates, P. \ Schneider and R. \ Sujatha, Modules over Iwasawa algebras. \emph{J. Inst. Math. Jussieu} \textbf{2} (2003), 73-108.
\bibitem{C}J. \ Coates, T. \ Fukaya, K. \ Kato, R. \ Sujatha, O. \ Venjakob, The GL2 main conjecture for elliptic
curves without complex multiplication, \emph{Publ. Math. IHES} \textbf{101} (2005), 163-208.
\bibitem{CM} J. \ C. \ McConnell, J. \ C. \ Robson, Noncommutative Noetherian Rings. \emph{ LMS Lecture Note Series 98} (1986).
\bibitem{O} O. \ Venjakob, A noncommutative Weierstrass Preparation Theorem and applications to Iwasawa Theory. \emph{ J. Reine Angew. Math} (2003), 153-191.
\bibitem{B} C. \ Weibel, K-book, Chapter-II, Chapter III,
\url{http://www.math.rutgers.edu/~weibel/Kbook}
\end{thebibliography}
\end{document}